\documentclass[12pt]{amsart}
 
\usepackage[dvips]{graphicx}
	\usepackage{amsmath,amssymb,amsthm,wrapfig,amsfonts,enumerate,latexsym}

	\newtheorem{dfn}{Definition}[section]
	\newtheorem{thm}[dfn]{Theorem}
	\newtheorem{prop}[dfn]{Proposition}
	\newtheorem{cor}[dfn]{Corollary}
	\newtheorem{lem}[dfn]{Lemma}
	\newtheorem{rem}[dfn]{Remark}

	\newtheorem{ack}{Acknowledgements\!\!}

	\newcounter{yon}
	\numberwithin{equation}{section}

	\newcommand{\dist}{\mathop{\mathit{d}} \nolimits}

	\newcommand{\pr}{\mathop{\mathrm{pr}} \nolimits}

	\newcommand{\ric}{\mathop{\mathit{Ric}}        \nolimits}

\begin{document}

	\title[Exponential and Gaussian concentration of $1$-Lipschitz maps]
    {Exponential and Gaussian concentration of $1$-Lipschitz maps}
	\author[Kei Funano]{Kei Funano}
	\address{Department of Mathematics and Engineering
Graduate School of
Science and Technology
Kumamoto University
Kumamoto, 860-8555, Japan}
	\email{yahoonitaikou@gmail.com}
	\thanks{This work was partially supported by Research Fellowships of
	the Japan Society for the Promotion of Science for Young Scientists.}
	\dedicatory{}
	\date{\today}

	\maketitle

%\tableofcontents

\begin{abstract}In this paper, we prove an exponential and Ganssian
 concentration inequality for $1$-Lipschitz maps from mm-spaces to
 Hadamard manifolds. In particular, we give a complete answer to a
 question by M. Gromov.
 \end{abstract}
	\setlength{\baselineskip}{5mm}

%    \tableofcontents
\section{Introduction and statement of the main result}
In this paper, we study the theory of the L\'{e}vy-Milman concentration of
$1$-Lipschitz maps from an mm-space (metric measure space) into an Hadamard manifold. 
An \emph{mm-space} $X=(X,\dist_X,\mu_X)$ is a complete separable metric
space $(X,\dist_X)$ with a Borel probability measure $\mu_X$. Let $N$ be an $m$-dimensional Hadamard
manifold, i.e., a complete simply-connected Riemannian manifold with
non-positive sectional curvature. Given a Borel measurable map $f:X\to N$
such that the push-forward measure $f_{\ast}(\mu_X)$ of $\mu_X$ by $f$ has the finite
moment of order $2$, we define its \emph{expectation} $\mathbb{E}(f)$ by the center of mass of the measure $f_{\ast}(\mu_X)$.

We shall consider a closed Riemannian manifold
$M$ as an mm-space with the volume measure $\mu_M$ normalized as $\mu_M(M)=1$. We denote by $\lambda_1(M)$ the first
	 non-zero eigenvalue of the Laplacian on $M$. In \cite[Section
 $3\frac{1}{2}.41$]{gromov}, M. Gromov proved that 
 \begin{align}\label{gs1}
  \mu_M(\{ x\in M \mid \dist_N(f(x),\mathbb{E}(f))\geq r\})\leq m/(\lambda_1(M)r^2)
  \end{align}for any $1$-Lipschitz map $f:M\to N$, where $N$ is any $m$-dimensional Hadamard manifold. He also
  asked in \cite[Section $3\frac{1}{2}.41$]{gromov} that if the
  right-hand side of the above
  inequality (\ref{gs1}) can be improved by the form $C_1e^{-C_2
  \sqrt{m/\lambda_1(M)}r}$ or the form
  $C_1e^{-C_2(m/\lambda_1(M))r^2}$. In this paper, we give an answer to
  this question affirmatively.

  To state our main result, we need some definition. We define the
  \emph{concentration function} $\alpha_X:(0,+\infty)\to \mathbb{R}$ of
  an mm-space $X$ as the supremum of
        $\mu_X(X\setminus A_{r})$, where $A$ runs over all Borel subsets
        of $X$ with $\mu_X(A)\geq 1/2$ and $A_{r}$ is an open
        $r$-neighbourhood of $A$. We shall consider an mm-space $X$ satisfying that
 \begin{align}\label{s1}
  \alpha_X(r)\leq C_Xe^{-c_Xr^p}
  \end{align}for any $r>0$ and some constants $c_X,C_X,p>0$. In the case of $p=1$ (resp., $p=2$), the space
  $X$ is said to have the \emph{exponential concentration} (resp., \emph{Gaussian concentration}). For example, a closed Riemannian
  manifold $M$ satisfies that $\alpha_M(r)\leq
  e^{-\sqrt{\lambda_1(M)}r/3}$ (\cite[Theorem 4.1]{milgro},
  \cite[Theorem 3.1]{ledoux}). If the manifold $M$ moreover satisfies
  that $\ric_M \geq \kappa >0$, then we have $\alpha_M(r)\leq
  e^{-\kappa r^2 /2}$ (\cite[Section 1.2, Remark
  2]{milgro}, \cite[Theorem 2.4]{ledoux}). For an mm-space satisfying
  (\ref{s1}) and $m\in \mathbb{N}$, we put
  \begin{align*}
  A_{m,X}:=  1+\frac{\sqrt{\pi}}{4}\max\{
   1,2C_X\}e^{2C_X + (m+1)/(4m-2)}\{2+e^{1/(4m-2)}\} 
  \end{align*}and
  \begin{align*}
  \widetilde{A}_{m,X}:=1+\frac{\sqrt{\pi}C_Xe^{(m+1)/(4m-2)}}{2} \{ 2+ e^{1/(4m-2)}\}.
  \end{align*}We also put
  \begin{align*}
   B_{m,X}:= 1+\frac{\sqrt{\pi}e^{(m+1)/(4m-2)}}{2} \max\{e^{(\pi
   C_X)^2/2}, 2C_X e^{(\pi C_X)^2}\}
   \end{align*}and
   \begin{align*}
    \widetilde{B}_{m,X}:=1+ \sqrt{\pi}C_X e^{(m+1)/(4m-2)}.
    \end{align*}
Our main result is the following.

 \begin{thm}\label{mth}Let an mm-space $X$ satisfy $(\ref{s1})$, $N$ be an
  $m$-dimensional Hadamard manifold, and $f:X \to N$ a $1$-Lipschitz
  map. Then, we have the following $(1)$ and $(2)$.

  \begin{enumerate}
   \item[$(1)$] If $p=1$, then,
  for any $r>0$,
  we have
  \begin{align}\label{s1.1}
   \mu_X(\{ x\in X\mid \dist_N(f(x),\mathbb{E}(f))\geq r\}) \leq \min\{
   A_{m,X}e^{-(c_X/\sqrt{2m})r} , \widetilde{A}_{m,X}e^{-(c_X /(2\sqrt{2m}))r} \}.
   \end{align}
                \item[$(2)$] If $p=2$, then, for any $r>0$, we have
  \begin{align}\label{s1.2}
   \mu_M(\{ x\in X \mid \dist_N(f(x),\mathbb{E}(f))\geq r\}) \leq \min\{
   B_{m,X} e^{-(c_X/ (8m))r^2}, \widetilde{B}_{m,X}e^{-(c_X / (16m))r^2}\}.
   \end{align}
   \end{enumerate}
  \end{thm}As a corollary of Theorem \ref{mth}, we obtain the
   following. For $m\in \mathbb{N}$, we put
    \begin{align*}
    A_m:=1+\frac{\sqrt{\pi}e^{(9m-3)/(4m-2)}}{2}\{ 2 + e^{1/(4m-2)}\}
    \text{ and }\widetilde{A}_m:= 1 +
    \frac{\sqrt{\pi}e^{(m+1)/(4m-2)}}{2}\{ 2+ e^{1/(4m-2)}\}.
    \end{align*}We also put
\begin{align*}
    B_m:= 1+\sqrt{\pi}e^{\pi^2 + (m+1)/(4m-2)} \text{ and
    }\widetilde{B}_m:= 1+\sqrt{\pi}e^{(m+1)/(4m-2)}.
    \end{align*}
  \begin{cor}\label{mc}Let $M$ be a closed Riemannian manifold, $N$ an
   $m$-dimensional Hadamard manifold, and $f:M\to N$ a $1$-Lipschitz
   map. Then, we have the following $(1)$ and $(2)$.
\begin{enumerate}
 \item[$(1)$] For any $r>0$, we have
   \begin{align}\label{ssss}
     &\mu_M(\{ x\in M \mid \dist_N(f(x),\mathbb{E}(f))\geq r\})\\ \leq \ &
    \min\{A_m e^{-3^{-1}\sqrt{\lambda_1(M)/(2m)} r}, \widetilde{A}_m
    e^{-6^{-1}\sqrt{\lambda_1(M)/(2m)}r} \} \tag*{}.
    \end{align}
 \item[$(2)$] If moreover
   $\ric_M\geq \kappa>0$ holds, then for any $r>0$ we also have
   \begin{align}\label{ssss1}
      &\mu_M(\{ x\in M \mid \dist_N(f(x),\mathbb{E}(f))\geq r\})\\ \leq
    \ &
     \min \{ B_{m}e^{-(\kappa / (16m))r^2} , \widetilde{B}_{m}e^{-(\kappa
    / (32 m))r^2}\}. \tag*{}
    \end{align}
   \end{enumerate}
   \end{cor}

   The inequality (\ref{ssss}) is sharper than the inequality
   (\ref{gs1}) if $r$ is
   large enough. In the case where $M=\mathbb{S}^n$, the inequality
   (\ref{ssss1}) is sharp in a sense (see Remark \ref{2354}). Theorem \ref{mth} and Corollary \ref{mc} answer the question by
   Gromov. To prove the theorem, we use a traditional method of the
   Gibbs-Laplace transform (see \cite[Section 1.6]{ledoux}), i.e., we estimate $\int_Xe^{\lambda
   \dist_N(f(x),\mathbb{E}(f))}d\mu_X(x)$ for $\lambda>0$ from above, and then
   substitute a suitable value to $\lambda$. To do this, we estimate
   $\int_X \dist_N(f(x), \mathbb{E}(f))^qd\mu_X(x)$ for $q\geq 1$ by using
   the method of M. Ledoux and K. Oleszkiewicz in \cite[Theorem
   1]{ledole}.

   For $m\leq n$, we consider
   the $m$-dimensional standard unit sphere
   $\mathbb{S}^m$ in $\mathbb{R}^{m+1}$ centered at zero as a subset of $\mathbb{S}^{n}$ in a
   natural way. As an application of Corollary \ref{mc}, we estimate
   $\mu_{\mathbb{S}^n}(\mathbb{S}^n \setminus (\mathbb{S}^m)_r)$ from
   above (Corollary \ref{egc4}). In \cite{artsein}, S. Artstein studied
   an asymptotic behavior of the values $\mu_{\mathbb{S}^n}(\mathbb{S}^n
   \setminus (\mathbb{S}^m)_r)$. We will compare our estimate with
   those Artstein's results (see Remark \ref{2354}). Denote by $\gamma_m$ the standard
  Gaussian measure on $\mathbb{R}^m$ with the density
  $(2\pi)^{-m/2}e^{-|x|^2/2}$. In \cite[Theorem 1]{ledole}, motivated by
  the work of Gromov (\cite{gromov2}), Ledoux and Oleszkiewicz obtained
    that if an mm-space having the Gaussian concentration (\ref{s1}),
    then for an $m$-dimensional Hadamard manifold $N$ and a $1$-Lipschitz map $f:X\to N$, we have
   \begin{align}\label{s1.4}
    \mu_X(\{ x\in X \mid \dist_N(f(x),\mathbb{E}(f))\geq r\})\leq C C_X
    \gamma_m ( \{  x\in \mathbb{R}^m \mid |x|\geq C \sqrt{c_X }r\}),
    \end{align}where $C>0$ is a universal constant. Their estimate (\ref{s1.4}) is
    highly relevant with our two estimate in Theorem \ref{mth}. We will compare these
    estimate (see Remark \ref{remrem}).

\section{Preliminaries}
\subsection{Concentration of 1-Lipschitz functions around the expetcations}
In this subsection we explain some basic facts on the theory of the L\'{e}vy-Milman
concentration of $1$-Lipschitz functions, which will be useful to prove
the main theorem. The theory of the concentration of $1$-Lipschitz functions was introduced by V. Milman in his
investigations of asymptotic geometric analysis (\cite{mil1,mil2,mil3}). 

Let $X$ be an mm-space and $f:X\to \mathbb{R}$ a Borel measurable
function. A number $m_f\in \mathbb{R}$ is called a \emph{median} of $f$
if it satisfies that $\mu_X(\{x\in X \mid f(x)\geq m_f\})\geq 1/2$ and
$\mu_X(\{x\in X \mid f(x)\leq m_f\})\geq 1/2$. We remark that $m_f$ does
exist, but it is not unique for $f$ in general.
\begin{lem}[{\cite[Section 1.3]{ledoux}}]\label{l2.1}Let $X$ be an mm-space. Then, for any $1$-Lipschitz
 function $f:X\to \mathbb{R}$ and median $m_f$ of $f$, we have
\begin{align*}
 \mu_X(\{ x\in X\mid |f(x)-m_f|\geq r\})\leq 2\alpha_X(r).
 \end{align*}Conversely, if a function $\alpha:(0,+\infty)\to
 [0,+\infty)$ satisfies that
 \begin{align*}
  \mu_X(\{  x\in X \mid |f(x)-m_f|\geq r\}) \leq \alpha(r)
  \end{align*}for any $1$-Lipschitz function $f:X\to \mathbb{R}$ and median $m_f$ of $f$,
 then we have
 \begin{align*}
   \alpha_X(r)\leq \alpha(r).
 \end{align*}
 \end{lem}

  Although the
   following lemma is stated in \cite{ledoux}, we prove them for the
   completeness of this paper. Given $p>0$, we put $K_p:= \int_0^{+\infty}e^{-r^p}dr=\frac{1}{p}\Gamma\big(\frac{1}{p}\big)$.
\begin{lem}[{cf.~\cite[Proposition 1.8]{ledoux}}]\label{l2.3}Assume that an mm-space
 $X$ satisfies (\ref{s1}). Then, for any $p\geq 1$ and any $1$-Lipschitz function $f:X\to \mathbb{R}$ with
 expectation zero, we have
 \begin{align*}
  \mu_X( \{x\in X \mid |f(x)|\geq r \})\leq \max \{e^{2(C_X K_p)^p },2C_Xe^{(2 C_X K_p)^p} \}e^{-2^{1-p}c_Xr^p}.
 \end{align*}
 \begin{proof}By virtue of Lemma \ref{l2.1}, we have
  \begin{align}\label{s2.1}
   \mu_X(\{ x\in X \mid |f(x)-m_f|\geq r\})\leq 2C_Xe^{-c_Xr^p}
   \end{align}for any $r>0$. By using this, we calculate
  \begin{align}\label{s2.2}
   |m_f |\leq \ &\int_X|
   f(x)-m_f|d\mu_X(x)\\ \tag*{} \leq \ &\int_0^{+\infty}\mu_X(\{x\in X \mid |f(x)-m_f|\geq
   r\}) dr\\ \tag*{}
   \leq \ &2C_X\int_0^{+\infty}e^{-c_Xr^p}dr\\ \tag*{} = \ &\frac{2C_XK_p}{(c_X)^{1/p}}=:\overline{\alpha}
   \end{align}If $r>\overline{\alpha}$, combining (\ref{s2.1}) with (\ref{s2.2}), we then get
  \begin{align*}
   \mu_X(\{x\in X \mid |f(x)|\geq r \})\leq 2C_Xe^{-c_X
   (r-\overline{\alpha})^p}\leq \ &2C_Xe^{-c_X 2^{1-p}r^p
   +c_X\overline{\alpha}^p}\\ \leq \
   &2C_Xe^{(2C_XK_p)^p}e^{-c_X2^{1-p}r^p}. \tag*{}
   \end{align*}If $r\leq \overline{\alpha}$, we then obtain
  \begin{align*}
   \mu_X(\{x\in X \mid |f(x)|\geq r\})\leq 1 =
   e^{2(C_XK_p)^p}e^{-2(C_XK_p)^p}=\ &
   e^{2(C_XK_p)^p}e^{-2^{1-p}c_X\overline{\alpha}^p}\\ \leq \ &
   e^{2(C_XK_p)^p}e^{-2^{1-p}c_Xr^p}. \tag*{}
   \end{align*}This completes the proof.
  \end{proof}
 \end{lem}

 \subsection{Expectation of a map to an Hadamard manifold}

 In this subsection we define the expectation of a Borel measurable map
 from an mm-space to an Hadamard manifold. In order to define the expectation, we first
 explain some basic facts on the barycenter of a Borel probability measure on an Hadamard manifold. 

 Let $N$ be an Hadamard manifold. We denote by $\mathcal{P}^2(N)$ the
 set of all Borel probability measure $\nu $ on $N$
 having the finite moment of order $2$, i.e.,
 \begin{align*}
  \int_N \dist_N(x,y)^2d\nu(y)<+\infty
  \end{align*}for some (hence all) $x\in N$. A point $x_0\in N$ is
  called the \emph{barycenter} of a measure $\nu\in \mathcal{P}^2(N)$ if $x_0$ is the
  unique minimizing point of the function
  \begin{align*}
   N \ni x\mapsto \int_N \dist_N(x,y)^2d\nu(y)\in \mathbb{R}.
   \end{align*}We denote the point $x_0$ by $b(\nu)$. It is well-known
   that every $\nu \in \mathcal{P}^2(N)$ has the barycenter
   (\cite[Proposition 4.3]{sturm}).

A simple variational argument implies the following two lemmas.
\begin{lem}[{cf.~\cite[Proposition 5.4]{sturm}}]\label{bl1}For each
 $\nu\in \mathcal{P}^2(\mathbb{R}^m)$, we have 
 \begin{align*}
  b (\nu)= \int_{\mathbb{R}^m} y  d\nu(y).
  \end{align*}
\end{lem}
 \begin{lem}[{cf.~\cite[Proposition 5.10]{sturm}}]\label{bl2}Let $N$ be an Hadamard
  manifold and $\nu\in \mathcal{P}^2(N)$. Then $x= b(\nu)$ if and only if
  \begin{align*}
   \int_N \exp^{-1}_x (y) d\nu(y)=0.
   \end{align*}In particular, identifying the tangent space of $N$ at
  $b(\nu)$ with the Euclidean space of the same dimension, we have $b ((\exp^{-1}_{b(\nu)})_{\ast}(\nu))=0$.
  \end{lem}

     Let $f:X\to N$ be a Borel measurable map from an mm-space $X$ to an
     Hadamard manifold $N$ satisfying $f_{\ast}(\mu_X)\in
     \mathcal{P}^2(N)$. We define the \emph{expectation}
     $\mathbb{E}(f)\in N$ of the map $f$ by the point
$b(f_{\ast}(\mu_X))$. By Lemma \ref{bl1}, in the case where $N$ is a Euclidean
     space, this definition
coincides with the classical one:
\begin{align*}
 \mathbb{E}(f)=\int_{X}f(x)d\mu_X(x).
 \end{align*}
           \section{Proof of the main theorem}
           Let $X$ be an mm-space satisfying (\ref{s1}) and $f:X\to
 \mathbb{R}^m$ a $1$-Lipschitz map with expectation zero.
 To prove the main theorem, we shall estimate $V_q(f):=(\int_X|f(x)|^q
 d\mu_X(x))^{1/q}$ and $\widetilde{V}_q(f):= (\int_{X\times X} |f(x)-f(y)|^q
 d(\mu_{X}\times \mu_X)(x,y))^{1/q}$ for $q\geq 1$. We show Ledoux and Oleskiewicsz's argument in
 \cite[Theorem 1]{ledole} as follows.

 Let $\varphi:X\to \mathbb{R}$ be an arbitrary $1$-Lipschitz
       function with expectation zero and $q\geq 1$. For any $\alpha>-1$, we put
      \begin{align*}
       M_{\alpha}:= \int_{\mathbb{R}}|s|^{\alpha}  d\gamma_1 (s) =
       2^{\alpha /2}\pi^{-1/2}\Gamma\Big(\frac{\alpha +1}{2}\Big).
       \end{align*}By virtue of Lemma \ref{l2.3}, we
       obtain $ \mu_X(\{ x\in X \mid |\varphi(x)|\geq r\})\leq
       C_1e^{-C_2r^p}$, where both $C_1$ and $C_2$ are defined by
       \begin{align*}C_1:=\max\{ e^{2(C_X K_p)^p},
        2C_Xe^{(2C_XK_p)^p}     \} \text{ and } C_2:= 2^{1-p}c_X.
        \end{align*}We calculate that
        \begin{align}\label{egs9}
         \int_X |\varphi(x)|^qd\mu_X(x)= \ &\int_{0}^{+\infty} \mu_X( \{
        x\in X \mid |\varphi(x)|\geq r\})  d (r^q) \\
        \leq \ &C_1
        \int_{0}^{+\infty}e^{-C_2r^p}  d(r^q) \tag*{} \\ = \ &\frac{\sqrt{2\pi}qC_1M_{\frac{2q}{p}-1}}{p(2C_2)^{q/p}}. \tag*{}
         \end{align}
         Given any $1$-Lipschitz map $f:X\to \mathbb{R}^m$ with expectation zero and $z\in \mathbb{R}^m$, the map
        $z\cdot f:X\to \mathbb{R}$ is the $|z|$-Lipschitz function with
        expectation zero. By using the inequality (\ref{egs9}), we hence have
        \begin{align*}
         V_q(f)^q= \ &
         \int_{X}\Big\{\frac{1}{M_q}\int_{\mathbb{R}^m}|z\cdot f(x)|^q 
         d\gamma_m(z)\Big\}d\mu_X(x)  \\
         \leq \ &   \frac{\sqrt{2\pi}
         qC_1M_{\frac{2q}{p}-1}}{p(2C_2)^{q/p}M_q} \int_{\mathbb{R}^m}
         |z|^qd\gamma_m(z) \tag*{} 
         \end{align*}We therefore obtain
           \begin{align}\label{mosber5}
            V_q(f)^q\leq  \frac{2^{- (q/p) +
         (q/2)}\sqrt{\pi}\max\{
         e^{2(C_XK_p)^p},2C_Xe^{(2C_XK_p)^p}\}}{(c_X)^{q/p}}\cdot \frac{q\Gamma\big(\frac{q}{p}\big)}{p\Gamma\big(\frac{q+1}{2}\big)}\int_{\mathbb{R}^m}|z|^qd\gamma_m(z).
           \end{align}

 To get another estimate, we repeat the above argument by using the following lemma.

       \begin{lem}[{cf.~\cite[Corollary 1.5]{ledoux}}]\label{mosber1}Let $X$ be an
        mm-space and $\varphi:X\to \mathbb{R}$ a $1$-Lipschitz
        function. Then, for any $r>0$, we have
        \begin{align*}
         (\mu_X \times \mu_X)(\{ (x,y)\in X\times X \mid |\varphi(x)-\varphi(y)|\geq
         r\})\leq 2\alpha_X(r/2).
         \end{align*}
        \end{lem}

        Let $X$, $\varphi:X\to
        \mathbb{R}$, and $f:X\to \mathbb{R}^m$ be as above. By Lemma \ref{mosber1}, we calculate that
        \begin{align*}
        \widetilde{V}_q(\varphi)^q = \ &\int_{0}^{+\infty} (\mu_X \times \mu_X) ( \{
        (x,y)\in X \times X\mid |\varphi(x)-\varphi (y)|\geq r\})  d (r^q) \\
        \leq \ &2 C_X
        \int_{0}^{+\infty}e^{-c_X2^{-p}r^p}  d(r^q) \tag*{} \\ = \ &\frac{\sqrt{\pi}q2^{q+(3/2)-(q/p)}C_XM_{\frac{2q}{p}-1}}{p(c_X)^{q/p}}. \tag*{}
         \end{align*}We hence get
        \begin{align}\label{mosber2}
        \widetilde{V}_q(f)^q= \ &
         \int_{X\times X}\Big\{\frac{1}{M_q}\int_{\mathbb{R}^m}|z\cdot (f(x)-f(y))|^q 
         d\gamma_m(z)\Big\}d(\mu_X \times \mu_X)(x,y)  \\
         \leq \ &   \frac{\sqrt{\pi}
         q2^{q+(3/2)-(q/p)}C_XM_{\frac{2q}{p}-1}}{p(c_X)^{q/p}M_q} \int_{\mathbb{R}^m}
         |z|^qd\gamma_m(z) \tag*{} \\
         = \ & \frac{\sqrt{\pi}
         2^{(q/2)+1}C_X}{(c_X)^{q/p}} \cdot \frac{q\Gamma \big(\frac{q}{p}\big)}{p\Gamma
           \big(\frac{q+1}{2}\big)}
         \int_{\mathbb{R}^m}
         |z|^qd\gamma_m(z) \tag*{}.
         \end{align}
   Since $V_q(f)\leq \widetilde{V}_q(f)$, we therefore obtain
          \begin{align}\label{mosber3}
           V_q(f)^q\leq \frac{\sqrt{\pi}
         2^{(q/2)+1}C_X}{(c_X)^{q/p}} \cdot \frac{q\Gamma \big(\frac{q}{p}\big)}{p\Gamma
           \big(\frac{q+1}{2}\big)}
         \int_{\mathbb{R}^m}
         |z|^qd\gamma_m(z).
           \end{align}
\begin{rem}\label{egrr}\upshape We shall compare the inequality
 (\ref{mosber5}) with the inequality (\ref{mosber3}). For fixed $p, c_X, C_X$, the inequality
 (\ref{mosber3}) is worse than the inequality (\ref{mosber5}) if $q$ is large enough. If we fix $q,c_X$, then the inequality (\ref{mosber3}) is sharper than
 the inequality (\ref{mosber5}) if $p$ or $C_X$ is
 large enough. 
 \end{rem}

 We next explain the following observation by Gromov.
  \begin{prop}[{cf.~\cite[Section 13]{gromovcat}}]\label{p3.1}Let
   $f:X\to N$ be a $1$-Lipschitz map from an mm-space $X$ to an
   $m$-dimensional Hadamard manifold such that $f_{\ast}(\mu_X)\in \mathcal{P}^2(N)$. We identify the tangent space at $\mathbb{E}(f)$ with the Euclidean
   space $\mathbb{R}^m$ and consider the map
   $f_0:=\exp_{\mathbb{E}(f)}^{-1}\circ f:X\to \mathbb{R}^m$. Then, the
   map $f_0$ is a $1$-Lipschitz map with expectation zero satisfying that
\begin{align}\label{miki1}
 \mu_X(\{ x\in X \mid \dist_N(f(x), \mathbb{E}(f))\geq r\})= \mu_X(\{
 x\in X \mid |f_0(x)|\geq r\})
 \end{align}for any $r>0$.
   \begin{proof}The $1$-Lipschitz continuity of the map $f_0$ follows
    from Toponogov's comparison theorem. By Lemma \ref{bl2}, the expectation of the map
    $f_0$ is zero. Since the map $\exp^{-1}_{\mathbb{E}(f)}$ is
    isometric on rays issuing from $\mathbb{E}(f)$, we obtain
    (\ref{miki1}). This completes the proof.
    \end{proof}
   \end{prop}

           \begin{proof}[Proof of Theorem $\ref{mth}$]According to Proposition \ref{p3.1}, we only prove the case of
            $N=\mathbb{R}^m$. For $p=1,2$, we put
 \begin{align*}
  (C_1,C_2):= \Big(\frac{2\sqrt{\pi}C_X}{p} , \frac{\sqrt{2}}{(c_X)^{1/p}}\Big)
             \text{ or }\Big(\frac{\sqrt{\pi}}{p}\max\{ e^{2(C_X
             K_p)^p}, 2C_Xe^{(2C_XK_p)^p}\},
             \frac{2^{-(1/p)+(1/2)}}{(c_X)^{1/p}} \Big).
             \end{align*}Let $f:X\to \mathbb{R}^m$ be an arbitrary
            $1$-Lipschitz map with expectation zero.

            Assuming that $p=1$, we first prove the inequality
 (\ref{s1.1}). According to the inequalities (\ref{mosber5}) and (\ref{mosber3}), for $\lambda>0$, we estimate
            \begin{align*}
             \int_Xe^{\lambda |f(x)|}d\mu_X(x)=\
             &1+ \sum_{k=1}^{\infty}\frac{\lambda^k}{k!}V_k(f)^k\\
             \leq \ & 1+C_1C_2\lambda
             \sum_{k=1}^{\infty}\frac{(C_2\lambda)^{k-1}}{\Gamma
             \big(\frac{k+1}{2}\big)}\int_{\mathbb{R}^m}|z|^kd\gamma_m(z)\\
             = \ & 1+C_1C_2\lambda
             \int_{\mathbb{R}^m}|z|\sum_{k=0}^{\infty}\frac{(C_2\lambda |z|)^k}{\Gamma
             \big(\frac{k}{2}+1\big)} d\gamma_m(z).
             \end{align*}
            Since
            \begin{align*}
             \sum_{k=0}^{\infty}\frac{(C_2\lambda |z|)^k}{\Gamma
             \big(\frac{k}{2}+1\big)}=
             \sum_{k=0}^{\infty}\frac{(C_2\lambda |z|)^{2k}}{\Gamma
             (k+1)} + \sum_{k=0}^{\infty}\frac{(C_2\lambda
             |z|)^{2k+1}}{\Gamma \big(\frac{2k+1}{2}+1\big)} \leq
             e^{(C_2\lambda)^2 |z|^2}+C_2\lambda |z|e^{(C_2\lambda)^2|z|^2},
             \end{align*}we thus get
            \begin{align*}
             &\int_{X}e^{\lambda |f(x)|}d\mu_X(x)\\
             \leq \ &
             1+ C_1C_2\lambda
             \int_{\mathbb{R}^m}|z|e^{(C_2\lambda)^2|z|^2}d\gamma_m(z) +
             C_1(C_2\lambda)^2\int_{\mathbb{R}^m} |z|^2
             e^{(C_2\lambda)^2|z|^2} d\gamma_m(z).
             \end{align*}Assume that $2(C_2\lambda)^2<1$. Then, we have
            \begin{align*}
             \int_{X}e^{\lambda|f(x)|}d\mu_X(x) \leq
             1+C_1C_2\lambda
             (1-2(C_2\lambda)^2)^{-(m+1)/2}\int_{\mathbb{R}^m}|z|d\gamma_m(z)\\
             +C_1(C_2\lambda)^2(1-2(C_2\lambda)^2)^{-(m/2)-1}\int_{\mathbb{R}^m}|z|^2d\gamma_m(z).
             \end{align*}
            By using the Chebyshev inequality, we hence have
            \begin{align*}
             \mu_X(\{ x\in X \mid |f(x)|\geq
             r\})\leq\ & e^{-\lambda r}
             \int_{X}e^{\lambda |f(x)|}d\mu_X(x)\\
             \leq \ &
             e^{-\lambda r}\Big\{1+C_1C_2\lambda
             (1-2(C_2\lambda)^2)^{-(m+1)/2}\int_{\mathbb{R}^m}|z|d\gamma_m(z)\\
              &\hspace{0.5cm}+C_1(C_2\lambda)^2(1-2(C_2\lambda)^2)^{-(m/2)-1}\int_{\mathbb{R}^m}|z|^2d\gamma_m(z)
             \Big\}.
             \end{align*}Substituting $\lambda:=1/(2C_2\sqrt{m})$ to
            this inequality, we therefore obtain
            \begin{align}\label{egs11}
             &\mu_X(\{ x\in X \mid |f(x)|\geq r\}\\ \leq \ &
             e^{-r/(2C_2\sqrt{m}) }\Big\{1+\frac{C_1}{2\sqrt{m}}
             \Big(1-\frac{1}{2m}\Big)^{-(m+1)/2}\int_{\mathbb{R}^m}|z|d\gamma_m(z)
             \tag*{} \\
            &  \hspace{6cm}+\frac{C_1}{4m}\Big(1-\frac{1}{2m}\Big)^{-(m/2)-1}\int_{\mathbb{R}^m}|z|^2d\gamma_m(z)
             \Big\}. \tag*{}
             \end{align}Observe that 
            \begin{align}\label{egs12}
             \int_{\mathbb{R}^m}|z|^2d\gamma_m(z)=m,
             \int_{\mathbb{R}^m}|z|d\gamma_m(z)\leq \Big( \int_{\mathbb{R}^m}|z|^2d\gamma_m(z)\Big)^{1/2}=\sqrt{m},
             \end{align}and $(1+1/x)^x\leq e$ for all $x>0$. Applying
            these to (\ref{egs11}), we obtain the inequality (\ref{s1.1}).

            Assume that $p=2$. We next prove (\ref{s1.2}) in a similar way
            to the above proof. By virtue of the inequalities (\ref{mosber5})
 and (\ref{mosber3}), given $\lambda
            >0$, we have
            \begin{align*}
             \int_X e^{\lambda |f(x)|}
             d\mu_X(x)=\ & 1+\sum_{k=1}^{\infty}
             \frac{\lambda^k}{k!}V_k(f)^k\\
             \leq \ & 1 + C_1 C_2 \lambda \sum_{k=1}^{\infty} \frac{(C_2
             \lambda)^{k-1}}{(k-1)!}\int_{\mathbb{R}^m}
             |z|^kd\gamma_m(z)\\
             = \ & 1+ C_1 C_2 \lambda \int_{\mathbb{R}^m}|z|e^{C_2
             \lambda |z|}d\gamma_m(z).
             \end{align*}Since
            \begin{align*}
             C_2 \lambda |z| = (\sqrt{2m}C_2\lambda)\cdot
             \Big(\frac{|z|}{\sqrt{2m}}\Big) \leq m(C_2\lambda)^2 +
             \frac{|z|^2}{4m},
             \end{align*}we calculate that
            \begin{align*}
             \int_{\mathbb{R}^m}|z|e^{C_2\lambda |z|}d\gamma_m(z)\leq \
             &e^{m(C_2\lambda)^2}\int_{\mathbb{R}^m}|z|e^{\frac{|z|^2}{4m}}d\gamma_m(z)\\
             = \ & e^{m(C_2\lambda)^2} \Big(1-\frac{1}{2m}\Big)^{-\frac{m+1}{2}}\int_{\mathbb{R}^m}|z|d\gamma_m(z).
             \end{align*}
            We hence get
            \begin{align*}
             \mu_X(\{ x\in X \mid |f(x)|\geq
             r\})\leq \ &e^{-\lambda r}\int_X e^{\lambda
             |f(x)|}d\mu_X(x)\\
             \leq \ & e^{-\lambda r}\Big\{ 1+ C_1C_2 \lambda
             e^{m(C_2\lambda)^2}
             \Big(1-\frac{1}{2m}\Big)^{-\frac{m+1}{2}}\int_{\mathbb{R}^m}|z|d\gamma_m(z)
             \Big\}.
             \end{align*}Putting $\lambda:= sr/(\sqrt{m}C_2)$ for any
            $s>0$, we thus have the estimate
            \begin{align*} \mu_X(\{ x\in X \mid |f(x)|\geq
             r\})\leq \ & e^{-sr^2/(\sqrt{m}C_2)}\Big\{ 1+
             \frac{C_1}{\sqrt{m}} sre^{s^2r^2}
             \Big(1-\frac{1}{2m}\Big)^{-\frac{m+1}{2}}\int_{\mathbb{R}^m}|z|d\gamma_m(z)\Big\}\\
             \leq \ & e^{-sr^2/(\sqrt{m}C_2)}\Big\{ 1+
             \frac{C_1}{\sqrt{m}} e^{2s^2r^2}
             \Big(1-\frac{1}{2m}\Big)^{-\frac{m+1}{2}}\int_{\mathbb{R}^m}|z|d\gamma_m(z)\Big\}
             \end{align*}Substituting $s:=1/(4\sqrt{m}C_2)$ into this
            inequality, we calculate that
            \begin{align*}
             \mu_X(\{x\in X\mid |f(x)|\geq r
             \})\leq 
             e^{-r^2/(8m(C_2)^2)}\Big\{1+\frac{C_1}{\sqrt{m}}\Big(1-\frac{1}{2m}\Big)^{-\frac{m+1}{2}}
             \int_{\mathbb{R}^m}|z|d\gamma_m(z)\Big\}.
             \end{align*}As a consequence, by (\ref{egs12}), we obtain
            the inequality (\ref{s1.2}). This completes the proof.
            \end{proof}

            \section{Applications and remarks}

In this section, we obtain two applications of Corollary \ref{mc} and compare
our results with the results by S. Artstein \cite{artsein} and Ledoux
and Oleszkiewicz \cite{ledole}. 
         \begin{cor}\label{egc4}Let $m\leq n$. Then, for any $r>0$, we have
    \begin{align}\label{egs15}
     &\mu_{\mathbb{S}^n}(\mathbb{S}^n \setminus (\mathbb{S}^{n-m})_r)
     \leq 
     \min\{A_m e^{-(1/(3\pi))\sqrt{2n/m}r} , \widetilde{A}_m
     e^{-(1/(3\pi))\sqrt{n/(2m)}r}, 
     \\ & \hspace{6cm}B_me^{-((n-1)/(4\pi^2m))r^2}, \widetilde{B}_me^{-((n-1)/(8\pi^2m))r^2}\} \tag*{}.
     \end{align}
            \begin{proof}Applying Corollary \ref{mc} to the projection
             \begin{align*}
              \mathbb{S}^n\ni (x_1,x_2, \cdots
             x_{n+1})\mapsto (x_1,x_2,\cdots,x_m)\in \mathbb{R}^m
              ,\end{align*}
             we
             obtain (\ref{egs15}). This completes the proof.
             \end{proof}
          \end{cor}

             The following corollary is a consequence of the theorem of the
    isoperimetry of waists of the Euclidean sphere by Gromov
    (\cite[Section 1]{gromov2}) and the inequality (\ref{egs15}).

     \begin{cor}\label{egc5}Let
 $m$ and $n$ be two natural numbers such that $m\leq n$ and $f:\mathbb{S}^n
 \to \mathbb{R}^m$ a continuous map. Then, there exists a point
 $z_f\in \mathbb{R}^m$ such that
  \begin{align*}
  \mu_{\mathbb{S}^n}(\mathbb{S}^n \setminus
  &(f^{-1}(z_f))_{r})\leq  \min\{A_m e^{-(1/(3\pi))\sqrt{2n/m}r} , \widetilde{A}_m
     e^{-(1/(3\pi))\sqrt{n/(2m)}r}, 
     \\ & \hspace{6cm}B_me^{-((n-1)/(4\pi^2m))r^2}, \widetilde{B}_me^{-((n-1)/(8\pi^2m))r^2}\}
  \end{align*}for any $r>0$.
  \end{cor}

 Let us explain S. Artstein's results for the estimates of the values
 $\mu_{\mathbb{S}^n}(\mathbb{S}^n \setminus (\mathbb{S}^m)_r)$.
 
 For two variables $A$ and $B$ depending on $n$, $A\approx B$ means that
 $\lim_{n\to \infty} (A/B)=1$. Given
  $0<r <\pi/2$ and $0<\lambda <1$, we put
  \begin{align*}
   u(r,\lambda):= (1-\lambda)\log \frac{(1-\lambda)}{\sin^{2}r} +
   \lambda \log \frac{\lambda}{\cos^2 r}.
   \end{align*}Observe that $u(r,\lambda)\geq 0$ holds for all $r,\lambda$.

 \begin{thm}[{cf.~\cite[Theorem 3.1]{artsein}}]\label{egt6}For any $0<r<\pi/2$ and
  $0<\lambda <1$, the following estimates $(1)$ and $(2)$ both hold as $n\to \infty$. 
\begin{enumerate}
     \item[$(1)$] If $\sin^2 r >1-\lambda$, then we have
     \begin{align*}
      \mu_{\mathbb{S}^{n}}(\mathbb{S}^n \setminus (\mathbb{S}^{\lambda n})_{r}) \approx    \frac{1}{\sqrt{n\pi}} \frac{\sqrt{\lambda (1-\lambda)}}{\sin^2
      r -(1-\lambda)}e^{-\frac{n}{2}u(r,\lambda)}.
      \end{align*}

  \item[$(2)$] If $\sin^2 r <1-\lambda$, then we have
     \begin{align*}
      \mu_{\mathbb{S}^n}( \mathbb{S}^n \setminus
      (\mathbb{S}^{\lambda n})_{r}) \approx 1- \frac{1}{\sqrt{n\pi}} \frac{\sqrt{\lambda (1-\lambda)}}{\sin^2
      r -(1-\lambda)}e^{-\frac{n}{2}u(r,\lambda)}.
      \end{align*}
 \end{enumerate}
 \end{thm}
       \begin{thm}[{cf.~\cite[Theorem 4.1]{artsein}}]\label{egt8}Let $n\geq 6$, $3\leq m\leq
        n-3$, and $\lambda:= m/n$. Put
        \begin{align*}
         l:=\frac{\sin^2 r}{1-\lambda} \text{ and }l':= \frac{\cos^2 r}{\lambda}.
         \end{align*}Then there exist positive constants $c_{n,\lambda}$
        and $c_{n,\lambda}'$ both bounded from above by $3$ satisfying the
        following $(1)$ and $(2)$.

        \begin{enumerate}
        \item[$(1)$] If $\sin^2 r<1-\lambda$, then
        \begin{align*}
       \frac{1}{\sqrt{2\pi}}\frac{e^{-u-c_{n,\lambda}' - \log
         l'}}{\frac{1}{\sqrt{u+c_{n,\lambda}' +\log l'}}+
         \sqrt{u+c_{n,\lambda}' +\log l'}} \leq \mu_{\mathbb{S}^n}((\mathbb{S}^m)_r)\leq
         \frac{1}{\sqrt{2\pi}}\frac{e^{-u-c_{n,\lambda}-\log l}}{\sqrt{u+c_{n,\lambda}+\log l}}.
         \end{align*}

        \item[$(2)$] If $\sin^2 r >1-\lambda$, then
        \begin{align*}
         1-  \frac{1}{\sqrt{2\pi}}\frac{e^{-u-c_{n,\lambda}' - \log
         l'}}{\frac{1}{\sqrt{u+c_{n,\lambda}' +\log l'}}+
         \sqrt{u+c_{n,\lambda}' +\log l'}}
         \leq \mu_{\mathbb{S}^n}((\mathbb{S}^m)_r)\leq 1- \frac{1}{\sqrt{2\pi}}\frac{e^{-u-c_{n,\lambda}-\log l}}{\sqrt{u+c_{n,\lambda}+\log l}},
         \end{align*}where $u=\frac{n}{2}\big((1-\lambda )\log
        \frac{1-\lambda }{\sin^2 r}      +\lambda \log
        \frac{\lambda}{\cos^2 r}      \big)$.
         \end{enumerate}
        \end{thm}

        \begin{rem}\label{2354}\upshape Fix $0<\lambda<1$. By using Theorem \ref{egt6}
         or Theorem \ref{egt8}, we have $\lim_{n\to
         \infty}\mu_{\mathbb{S}^n}(\mathbb{S}^n \setminus
         (\mathbb{S}^{\lambda n})_r)=0$ for all $r
         >\sin^{-1} \sqrt{1-\lambda}$, which cannot be derived from
         Corollary \ref{egc4}. Theorem \ref{egt6} and \ref{egt8}
         therefore both contain
         some information for the values
         $\mu_{\mathbb{S}^n}((\mathbb{S}^m)_r)$ which Corollary \ref{egc4}
         does not contain. The author does not know how to derive
         Corollary \ref{egc4} from Theorems \ref{egt6} and \ref{egt8}. However, Corollary
         \ref{egc4} (and also the inequality (\ref{ssss1})) is sharp
         in the following sense. Denote by $\pr_n$ the projection from
         $\mathbb{S}^n(\sqrt{n})$ to the Euclidean space
         $\mathbb{R}^m$. Since the sequence $\{
         (\pr_n)_{\ast}(\mu_{\mathbb{S}^n(\sqrt{n})})\}_{n=1}^{\infty}$
         of probability measures on $\mathbb{R}^m$ weakly converges to
         the canonical Gaussian measure $\gamma_m$ on $\mathbb{R}^m$
         (see \cite[Lemma 1.2]{ledoux1}), by
         using the inequality (\ref{ssss1}) (or Corollary \ref{egc4}), we
         obtain the estimate
         \begin{align}\label{egs26}
          \gamma_m(\{x\in \mathbb{R}^m \mid |x|\geq r\})\leq \min\{B_m
          e^{-(1/(16m))r^2}, \widetilde{B}_me^{-(1/(32m))r^2} \}.
          \end{align}Classically, this inequality was known via an
         another method, see \cite[Section 3.1,
         (3.5)]{ledtal}. This estimate is sharp in a sense because
         \begin{align*}
          \lim_{r\to \infty}\frac{\gamma_m(\{   x\in \mathbb{R}^m \mid |x|\geq r
  \})}{e^{- (1/(2m))r^2}}=1
          \end{align*}(\cite[Theorem 3.8]{ledtal}).
         \end{rem}

         Let us compare our result with the inequality (\ref{s1.4}). 

          \begin{rem}\label{remrem}\upshape Combining the inequalities (\ref{s1.4}) with (\ref{egs26}), we obtain an estimate similar to the
           inequality (\ref{s1.1}).  However, we note that our coefficients of
           the inequality (\ref{s1.1}) are concrete whereas the coefficients
           of the inequality (\ref{s1.4}) are not. An advantage
           of the inequality (\ref{s1.4}) is that we can see
           from the inequality that the map $f$
           concentrates around the expectations if the coefficient $C_X$
           is close to zero. This fact cannot be derived from our
           inequality (\ref{s1.2}). We also remark that their
           proof cannot be applied to the case where $X$
           has the exponential concentration (\ref{s1}) (i.e., the
           case where $p=1$).
           \end{rem}

            \begin{ack}\upshape
        The author would like to express his thanks to
  Professor Takashi Shioya for his valuable suggestions and assistances
 during the preparation of this paper. 
 \end{ack}

	\end{document}